\documentclass[12pt]{amsart}
\usepackage{times}
\usepackage[T1]{fontenc}
\usepackage{amscd,verbatim,amssymb}
\usepackage[all]{xy}
\usepackage{color}
\usepackage[colorlinks,linkcolor=blue,citecolor=blue,urlcolor=red]{hyperref}

\renewcommand{\L}{\mathbb{L}}

\renewcommand{\P}{\mathbf{P}}
\newcommand{\Q}{\mathbf{Q}}

\newcommand{\Z}{\mathbf{Z}}

\newcommand{\sC}{\mathcal{C}}
\newcommand{\sD}{\mathcal{D}}

\newcommand{\sM}{\mathcal{M}}

\newcommand{\sO}{\mathcal{O}}

\newcommand{\un}{\mathbf{1}}

\renewcommand{\o}{\mathrm{o}}

\newcommand{\Sm}{\operatorname{\mathbf Sm}}

\newcommand{\alg}{{\operatorname{alg}}}
\renewcommand{\hom}{{\operatorname{hom}}}
\newcommand{\num}{{\operatorname{num}}}
\newcommand{\tnil}{{\operatorname{\otimes\mathrm{nil}}}}
\newcommand{\ab}{{\operatorname{ab}}}

\newcommand{\proj}{{\operatorname{proj}}}
\newcommand{\rat}{{\operatorname{rat}}}
\newcommand{\rep}{{\operatorname{rep}}}
\newcommand{\car}{\operatorname{char}}

\newcommand{\Corr}{\operatorname{Corr}}

\newcommand{\Pic}{\operatorname{Pic}}

\newcommand{\End}{\operatorname{End}}

\newcommand{\Coker}{\operatorname{Coker}}
\newcommand{\Ker}{\operatorname{Ker}}

\newcommand{\Griff}{\operatorname{Griff}}

\renewcommand{\phi}{\varphi}
\renewcommand{\epsilon}{\varepsilon}

\newcommand{\inj}{\hookrightarrow}

\newcommand{\surj}{\rightarrow\!\!\!\!\!\rightarrow}

\newcommand{\by}{\xrightarrow}

\newcommand{\iso}{\by{\sim}}

\renewcommand{\lim}{\varprojlim}

\newcounter{spec}
\newenvironment{thlist}{\begin{list}{\rm{(\roman{spec})}}%
{\usecounter{spec}\labelwidth=20pt\itemindent=0pt\labelsep=10pt}}%
{\end{list}}%

\newtheorem{thm}{Theorem}
\newtheorem{conj}{Conjecture}

\newtheorem{lemma}{Lemma}
\newtheorem{prop}{Proposition}
\newtheorem{cor}{Corollary}

\theoremstyle{definition}
\newtheorem{defn}{Definition}
\theoremstyle{remark}
\newtheorem{rque}{Remark}
\newtheorem{rques}{Remarks}
\newtheorem{ex}{Example}

\newtheorem{qn}{Question}
\newtheorem*{qns}{Questions}

\begin{document}
\title{Some remarks on the smash-nilpotence conjecture}
\author[B. Kahn]{Bruno Kahn}
\address{CNRS, Sorbonne Université and Université Paris Cité, IMJ-PRG\\ Case 247\\4 place
Jussieu\\75252 Paris Cedex 05\\France}
\email{bruno.kahn@imj-prg.fr}
\begin{abstract}
We discuss cases where Voevodsky's smash nilpotence conjecture is known, and give a few new ones. In particular we explain a theorem of the cube for $1$-cycles, which is due to Oussama Ouriachi.
\end{abstract}
\date{May 30, 2024}
\maketitle

\hfill {\it To the memory of Jacob Murre}

\subsection*{Introduction} Let $X$ be a smooth projective variety over a field $k$. An algebraic cycle $Z$ on $X$ is said to be \emph{smash-nilpotent} if there exists an integer $N>0$ such that $Z^N$ is rationally equivalent to $0$ over $X^N$. This notion is due to Voevodsky \cite{voe}, who proved:

\begin{thm}[\protect{\cite[Cor. 3.2]{voe}}]\label{t0}
Any cycle on $X$ (with rational coefficients) which is algebraically equivalent to $0$ is smash-nilpotent. 
\end{thm}

This was also proven independently by Voisin \cite[p. 267]{voi}. Their (identical) proof reduces to the case of $0$-cycles on a curve, by a classical trick due to Weil and Bloch.\footnote{Note that the statement makes sense and is true for any separated $k$-scheme of finite type, with the same proof.} For such a curve  $C$, we note that it is very close in spirit to Kimura's proof that the Chow motive $h_1(C)$ is oddly finite-dimensional \cite[Th. 4.2]{kim} (a result  originally due to Shermenev and reproven by Künnemann in \cite[Th. 3.3.1]{kunn}). Their argument is more elementary, using only $0$-cycles on symmetric powers of $C$, but actually follows from the finite dimensionality of $h_1(C)$ by \cite[Prop. 6.1]{kim}.


Voevodsky furthermore introduced the

\begin{conj}[\protect{\cite[Conj. 4.2]{voe}}]\label{c1} Any cycle numerically equivalent to $0$ is smash-nilpotent.
\end{conj}

This conjecture immediately implies that homological e\-quivalence coincides with  numerical equivalence  for any Weil cohomology, since a smash-nilpotent cycle is clearly homologically equivalent to $0$. In \cite[Prop. 4.6]{voe}, Voevodsky shows that it would follow from the existence of an abelian category of mixed motives with suitable properties. On the other hand, O'Sullivan proved in \cite[p. 7]{os} (see also \cite[11.5.3]{andre}):

\begin{thm}\label{tos} Conjecture \ref{c1} follows from Grothendieck's standard conjecture B plus the Bloch-Beilinson--Murre (BBM) conjectures \cite{jann3}.
\end{thm}

In turn, the BBM conjectures also follow from the existence of an abelian category of mixed motives \cite[Prop. 4.4 and Th. 5.2]{jann3}.

O'Sullivan's proof rests on a well-known lemma of Nori (\cite[proof of Prop. 5.3]{nori}, \cite[Prop. 4.8]{jannsen}) that we state here because we shall use it later:

\begin{lemma}\label{lnori} If $X$ is a connected smooth projective $k$-variety of dimension $d\ge 2n+1$, then for any $\alpha \in CH_n(X)_\Q$ there is  a projective embedding $X\inj \P$, a smooth linear section $i:Y \inj X$ of dimension $2n+1$ and a $\beta\in CH_n(Y)_\Q$ such that $\alpha=i_*\beta$. If $\alpha$ is homologically trivial for a Weil cohomology verifying the weak Lefschetz theorem, so is $\beta$.
\end{lemma}

When is Conjecture \ref{c1} known? As a consequence of Theorem \ref{t0}, it is true whenever numerical and algebraic equivalences agree, i.e. for the following types of cycles:

\begin{itemize}
\item cycles of dimension $0$;
\item cycles of codimension $1$ (Matsusaka's theorem).
\end{itemize}

The first  published example of cycles verifying Conjecture \ref{c1} without being algebraically equivalent to $0$ was given in \cite[Prop. 1]{k-s}: it holds for all ``skew cycles'' on an abelian variety.  The proof relies on \cite[Prop. 6.1]{kim} which, as we saw, can be used to prove Theorem \ref{t0}.\footnote{To explain this in motivic terms: any morphism of Chow motives $\L^i\to M$, where $\L$ is the Lefschetz motive, which factors through $h_j(A)\otimes \L^{\otimes i}$ for some abelian variety $A$ and $j$ odd, is smash-nilpotent; for $j=1$ this gives back Theorem \ref{t0}, and for general $j$ it gives \cite[Prop. 1]{k-s}.}
In particular, Conjecture \ref{c1} holds for abelian $3$-folds; this had already been shown by O'Sullivan in \cite[bot. p. 7]{os} by a different argument. I regret to have noticed it only recently.

Next, Sebastian proved Conjecture \ref{c1} in \cite{seb} for  $1$-cycles on products of curves, hence on abelian varieties, by an elaborate combinatorial argument. This argument was clarified in 2014 by Oussama Ouriachi (unfortunately unpublished), who found that it is a consequence of a ``theorem of the cube for $1$-cycles'' (Corollary \ref{c2} below).

 The main purpose of this note is to present Ouriachi's result, that we shall recover in a stronger form (Theorem \ref{t1} and Corollary \ref{c3}). I take the opportunity to list all tricks I know to obtain new cases of Voevodsky's conjecture out of old: the ``charm'' of working on this conjecture is that it encompasses a mix of categorical and geometric inputs, with no clear general picture (at least to me). This is used in Section \ref{s8} to recover several examples from the literature where Voevodsky's conjecture has been proven. Since this literature has become quite large, it would be tedious to check whether all of them can be recovered in this way (using, possibly, some nontrivial geometry); I try to formulate this precisely as a question in Section \ref{s9}.

I would like to dedicate this small text to Jacob Murre, whose own connection with the subject is clear via the BBM conjectures; Murre was also keenly interested in Theorem \ref{t0}. I thank O'Sullivan and the referee for several helpful comments.

\subsubsection*{Notation and conventions}
For simplicity, we assume the ground field $k$ algebraically closed: Conjecture \ref{c1} readily reduces to this case.
We work with effective motives with rational coefficients modulo an adequate equivalence relation $\sim$, whose category will be denoted by $\sM_\sim$. We adopt the covariant convention: the motive functor $h:\Sm^\proj\to \sM_\sim$ is covariant, and we write $M(n)=M\otimes \L^{\otimes n}$ for $M\in \sM_\sim$, where $\L$ is the Lefschetz motive. When $\sim$ is rational equivalence (Chow motives), we simply write $\sM_\rat=\sM$. If $X\in \Sm^\proj$, we write $A_\sim^n(X)= \sM_\sim(h(X),\L^n)$ (resp. $A^\sim_n(X)= \sM_\sim(\L^n,h(X))$ for the $\Q$-vector space of cycles of codimension (resp. dimension) $n$ on $X$, modulo $\sim$.

André was the first to explicitly observe that smash-nilpotence also defines an adequate equivalence: we shall denote it by $\tnil$.

\subsection{Schur finiteness and finite dimensionality} Before really starting this article, we mention an easy result.

In \cite{AK}, del Angel and Kimura extend the notions of finite dimensionality and Schur finiteness from objects to morphisms. We refer to Deligne \cite[\S 1]{dtens} for a review of Schur functors, and recall

\begin{defn}[\protect{\cite[Def. 1.10]{AK}}] Let $\sC$ be a $\Q$-linear symmetric monoidal category. A morphism $f : V \to W$ in $\sC$ is called evenly (resp. oddly) finite dimensional if $\Lambda^n(f)=0$ (resp. $S^n(f)=0$ for some $n\ge 0$; it is called finite dimensional if it is a sum of an evenly and an oddly finite dimensional morphism. The morphism $f$ is called Schur finite if $S_\lambda(f)= 0$ for some Young diagram $\lambda$.
\end{defn}

In view of the isomorphisms
\[V^{\otimes n} \simeq \bigoplus_{|\lambda|=n} S_\lambda(V)\]
where the sum  is over all Young diagrams of length $n$, a morphism $f$ is smash-nilpotent if and only if $S_\lambda(f)=0$ for \emph{all} $\lambda$ such that $|\lambda|=n$. Thus the following result is considerably weaker than Voeovdsky's conjecture for $1$-cycles.

\begin{prop} Any $n$-cycle $\alpha$ (numerically trivial or not) on a smooth projective variety $X$ is finite dimensional as a morphism in $\sM$: in fact, $\Lambda^2(\alpha)=0$.
\end{prop}

\begin{proof} This is obvious since $\Lambda^2(\L^n)=0$. (This drastic generalisation and trivialisation of my initial statement and proof were pointed out by O'Sullivan.)
\end{proof}

\subsection{Effectivity} For a smash-nilpotent cycle $\alpha$ on a smooth projective variety $X$, write $\exp(\alpha)$ for the smallest integer $N$ such that $\alpha^{\otimes(N+1)}=0$. An interesting question is the following: let $S\subseteq CH_*(X)_\Q$ be a set of smash-nilpotent algebraic cycles. Is there an integer $N>0$ such that $\exp(\alpha)\le N$ for all $\alpha\in S$? As observed in \cite[Rem. 2]{k-s}, the answer is positive for skew cycles on an abelian variety, with explicit bounds given by its Betti numbers; by the same argument, it is also positive for $\alpha\in \Pic^0(X)_\Q$ with $\exp(\alpha)\le 2\dim \Pic^0_X$.

The next result is more difficult. For $\alpha\in CH_0(X)_0$, write $h(\alpha)$ for the smallest integer $n$ such that $\alpha$ is rationally equivalent to a difference $\alpha_1-\alpha_0$, where $\alpha_0$ and $\alpha_1$ are effective $0$-cycles of degree $n$: this is the \emph{height} of $\alpha$.  Exceptionally, we don't assume $k$ algebraically closed.

\begin{prop}\label{p1} For any $n>0$, there exists an integer $c_n$ such that $h(\alpha)\le n$ $\Rightarrow$ $\exp(\alpha)\le c_n$.
\end{prop}

\begin{proof} Let $K=k(X^{2n})$; the $2n$ projections $X^{2n}\to X$ provide $2n$ rational points of $X$ over $K$ (``independent Weil generic points"). By a variant of \cite[Lemma p. 56]{mumford}, choose an integral curve $C_n$ on $X_K$ passing through these points, and let $\tilde C_n$ be its normalisation. I claim that we can take $c_n=2g_n$, where $g_n$ is the arithmetic genus of $\tilde C_n$. Indeed, $C_n$ is geometrically irreducible since it has rational points, hence we may spread it to a closed subvariety $\sC_n$ of $U\times_k X$ such that the projection $\sC_n\to U$ is (proper and) flat with geometrically integral fibres, where $U$ is a suitable open subset of $X^{2n}$ (EGA IV$_3$, Th. 12.2.4).  If $h(\alpha)\le n$, by Chow's moving lemma we may write $\alpha$ as $\alpha_1-\alpha_0$ where $\alpha_0$ and $\alpha_1$ are effective $0$-cycles of degree $n$ with support contained in $U$. The fibre of $\sC_n$ at $(\alpha_0,\alpha_1)$ is an iintegral curve $C_n(\alpha)$ on $X$ such that $\alpha_0$ and $\alpha_1$ have support on $C_n(\alpha)$. 

Write $\tilde \sC_n$ for the normalisation of $\sC_n$, and $\tilde C_n(\alpha)$ for the fibre of $\tilde \sC_n$ at $(\alpha_0,\alpha_1)$. Then $\tilde C_n\to U$ is still projective, so we can apply \cite[Cor. III.9.13]{hart} and find that the arithmetic genus of $\tilde C_n(\alpha)$ is still $g_n$. Finally, let $C$ be its normalisation  (which is smooth since $k$ is algebraically closed); its arithmetic (hence geometric) genus $g$ is $\le g_n$\footnote{If $f:E\to D$ is a finite surjective morphism of integral curves, then $\sO_D\to f_*\sO_E$ is injective with cokernel supported on a finite closed subset of $D$, hence $H^1(D,\sO_D)\to H^1(E,\sO_E)$ is surjective.}. Pulling $\alpha$ back through the finite surjective morphism $C\to \tilde C_n(\alpha)\to C_n(\alpha)$ yields a $0$-cycle $\tilde \alpha$ of degree $0$ on $C$ which maps to $\alpha$, and $\tilde\alpha^{\otimes (2g+1)}=0$ implies $\alpha^{\otimes (2g_n+1)}=0$.
\end{proof}

\begin{qn}
In Proposition \ref{p1}, can one find $c_n$ independent of $n$?
\end{qn} 

\subsection{An abstract framework} The introduction suggests an approach to Voevodsky's conjecture by induction on the dimension or codimension of the cycles considered. To make this precise,  let $\sim \ge \sim'$ be two comparable adequate equivalence relations (the inequality means that $A_\sim(X)\surj A_{\sim'}(X)$ for all $X$). For $M\in \sM_\sim$ and $n\ge 0$, consider the conditions:
\begin{align*}
\sM_\sim(\L^n,M)&\iso \sM_{\sim'}(\L^n,M)\label{V} \tag{$V(M,n)$}\\
\sM_\sim(M,\L^n)&\iso \sM_{\sim'}(M,\L^n)\label{V*} \tag{$V^*(M,n)$}.
\end{align*}

If $M=h(X)$, we simply write $V(X,n)$ and $V^*(X,n)$. The following extends observations from the introduction.

\begin{lemma}\label{l3} a) $V(M,n)$ (resp. $V^*(M,n)$) holds for $M$ and $N$ if and only if it holds for $M\oplus N$.\\
b) $V(M,n) \iff V(M(1),n+1)$ and $V^*(M,n) \iff V^*(M(1),\allowbreak n+1)$.\\
Suppose that $\sim=\tnil$ and $\sim'=\num$. Then\\
c) $V(M,0)$ is true for any $M$.\\
d) $V^*(M,1)$ is true for any $M$.\\
In particular, Conjecture \ref{c1} holds if $\dim X\le 2$.
\qed
\end{lemma}


\subsection{Reduction to small dimension} Take $\sim=\tnil$ and $\sim'=\hom$, for a Weil cohomology satisfying the weak Lefschetz theorem. Ouriachi observed that Lemma \ref{lnori} reduces $V(-,n)$ to the case of smooth projective varieties of dimension $\le 2n+1$. This does not apply to Conjecture \ref{c1} unless one knows that homological and numerical equivalences coincide; this is the case for $n=1$ in characteristic $0$ by  \cite[Cor. 1]{lieb}. Let us give a proof in any characteristic. First a lemma:

\begin{lemma}\label{l2} Let $X$ be smooth projective of dimension $\ge 3$ and let $i:Y\subset X$ be a smooth hyperplane section. Then the restriction map on divisors modulo numerical equivalence is injective and its cokernel is killed by a power of $p$, where $p$ is the exponential characteristic of $k$.
\end{lemma}

\begin{proof} If $p=1$, this follows from \cite[Cor. 4.9 b) and Rem. 4.10]{SGA2}. Let us prove it when $p>1$ (this issue was raised in \cite[Proof of Th. 2]{essential}).\footnote{I thank Yves Laszlo for a discussion leading to this argument.} Using Weak Lefschetz for $l$-adic cohomology ($l\ne p$) and Matsusaka's theorem \cite{matsusaka}, we get the injectivity. Then, \cite[Cor. 4.9 b)]{SGA2} shows that $\Pic(X)\to \Pic(Y_m)$ is bijective for $m$ large enough, where $Y_m$ is a suitable infinitesimal thickening of $Y$. It remains to show that $\Coker(\Pic(Y_m)\by{f^*} \Pic(Y))$ is killed by a power of $p$, where $f:Y\inj Y_m$ is the closed immersion,. Let $G$ (resp. $G(m)$) be the étale sheaf of units of $Y$ (resp. of $Y_m$).  Then $f^* G(m)$ is an extension of $G$ by a sheaf of exponent $p^r$ for some $r\ge 0$, and we conclude with the long cohomology exact sequence.
\end{proof}

\begin{thm}\label{t5} Take $\sim=\tnil$ and $\sim'=\num$. Let $X\in \Sm^\proj$ be of dimension $\ge 3$. If $V(T,1)$ holds for all smooth threefold sections $T$ of $X$ by linear subspaces of all projective embeddings $X\inj \P$, then $V(X,1)$ holds. In particular, 
$V(M,1)$ holds for all $M\in \sM$ if and only if it holds for $h(T)$, for all $3$-dimensional $T\in \Sm^\proj$.
\end{thm}

(Of course this is optimal, since surfaces have trivial Griffiths groups.)

\begin{proof} Let $Z$ be a $1$-cycle on $X$. By Lemma \ref{lnori}, one may find  $T\subseteq X$ as in the statement such that $Z$ is rationally equivalent to a $1$-cycle $Z'$ with support in $T$; the point is to show that, if $Z$ is numerically equivalent to $0$, we can choose $Z'$ numerically equivalent to $0$. Let $p_X$ (resp. $p_T$) be a projector on $X$ (resp. on $T$) as in \cite[\S 2]{essential}:  I claim that we can choose $p_X$ and $p_T$ ``compatible''. Namely, choose a lift $(D_1,\dots D_r)$ in $CH^1(X)$ of a basis  of $A^1_\num(X)$, whence by Lemma \ref{l2} a lift $(D'_1,\dots D'_r)$ in $CH^1(T)$ of the corresponding basis  of $A^1_\num(T)$, with $D'_i=i^*D_i$. If $(C'_1,\dots, C'_r)$ is a lift in $CH_1(T)$ of the dual basis of $A_1^\num(T)$, then $(C_1,\dots, C_r)$ is a lift in $CH_1(X)$ of the dual basis of $A_1^\num(X)$, with $C_i =i_* C'_i$. Now the projectors
\[p_X= \sum D_i\otimes C_i\in CH^{\dim X}(X\times X), \quad p_T = \sum D'_i\otimes C'_i\in CH^{3}(T\times T)\]
verify $i_*\circ p_T = p_X\circ i_*$.

If $Z$  is numerically equivalent to $0$, then 
\[Z=(1-p_X)Z=i_*(1-p_T)Z'\] 
where $(1-p_T)Z'$ is numerically equivalent to $0$, so we are done.
\end{proof}

\begin{rque} Suppose that a threefold $T$ has a Chow-Künneth decomposition verifying the lift of the standard conjecture B to rational equivalence given in \cite[Hyp. 7.1]{tunis}; this happens in several cases, namely abelian threefolds, products of a curve and a surface, complete intersections -- see \cite[Prop. 7.2]{tunis}. By loc. cit., Th. 7.7, we have an isomorphism
\[\Griff(T)=\Ker(A_1^\alg(T)\to A_1^\num(T)) = \sM_\alg(\L,t_3(T))\]
where $\Griff(T)$ is the (numerical) Griffiths group of $T$ and $t_3(T)$, a direct summand of $h_3(T)$, is the ``transcendental part'' of $h(T)$. Thus Conjecture \ref{c1} for $T$ would follow from the finite dimensionality of $t_3(T)$ -- which is known only when $T$ is of abelian type (see \S \ref{s5.1})\dots\ In particular, to the best of my knowledge Conjecture \ref{c1} remains open even for the product of a curve and a general surface or for a general hypersurface (see also the questions in Section \ref{s9}).
\end{rque}

\subsection{An unconditional version of Ouriachi's theorem} 
Let us say that a motive $M\in \sM_\sim$ is \emph{reduced} if $\sM_\sim(\un,M)=0$. If $M=h(X)$ for $X\in \Sm^\proj$ connected, the choice of a rational point yields a splitting $h(X)\simeq \un \oplus h^+(X)$ where $h^+(X)$ is reduced. This extends to any $M\in \sM$, yielding functorial split exact sequences
\[0\to M^+\to M\to a(M)\to 0\]
where $a(M)$ is an Artin motive, i.e. (since $k$ is algebraically closed) a sum of copies of $\un$ (see \cite[Prop. 2.2]{tunis}).

\begin{lemma}\label{l1} If $\sim\le \alg$, one has $\sM_\sim(\un, M^+)=0$ for any $M$; hence the decomposition $M\simeq M^+ \oplus a(M)$ is unique and functorial.
\end{lemma}

\begin{proof} We reduce to $M=h(X)$ with $X$ connected; this is then equivalent to the isomorphism $A_0^\alg(X)=\Q$.
\end{proof}

\begin{thm}\label{t1} Let $M_1,M_2,M_3,N\in \sM$. Then
\begin{align*}
\sM_\tnil(\L,N\otimes M^+_1\otimes M^+_2 \otimes M^+_3)&=0.
\end{align*}
\end{thm}

\begin{proof}  We have a basic isomorphism, obtained by reducing to the case $N=h(X)$, $M_1=h(Y)$:
\begin{equation}\label{eq1}
(N\otimes M_1)^+\simeq N^+ \oplus N\otimes M_1^+.
\end{equation}

It reduces us to the case $N=1$. We further reduce to $M_i=h(X_i)$ for $X_i\in \Sm^\proj$.

Write $X=X_1\times X_2 \times X_3$. Any cycle in $\sM_\tnil(\L,h^+(X_1)\otimes h^+(X_2) \otimes h^+(X_3)$ is the image of a cycle in $\sM_\tnil(\L,h(X))=A_1^\tnil(X)$ under the projection $h(X)=h(X_1)\otimes  h(X_2) \otimes h(X_3)\to h^+(X_1)\otimes  h^+(X_2) \otimes h^+(X_3)$. Thus, it suffices to prove the result for the image of any $[C]$, where $C$ is a curve traced on $X$.

Let $\tilde C$ be the normalisation of $C$. The morphism $\tilde C\to C\to X$ factors through a morphism
\[\tilde C\by{\Delta} \tilde C^3\by{(\pi_i)} X_1\times X_2\times X_3\]
where $\Delta$ is the diagonal embedding and $\pi_i:\tilde C\to X_i$ is the composition $\tilde C\to C\to X\to X_i$. By Lemma \ref{l1}, the diagram
\[\begin{CD}
\sM_\tnil(\L,h(\tilde C^3))@>>> \sM_\tnil(\L,h(X))\\
@VVV @VVV\\
\sM_\tnil(\L,h^+(\tilde C)^{\otimes 3}) @>>> \sM_\tnil(\L,h^+(X_1)\otimes h^+(X_2)\otimes \otimes h^+(X_3))
\end{CD}\]
commutes. Hence it suffices to show that the bottom left group is $0$. But $h^+(\tilde C)=h_1(\tilde C)\oplus \L$, hence
\[h^+(\tilde C)^{\otimes 3}=h_1(\tilde C)^{\otimes 3}\oplus M\otimes \L\]
where $M_{\le 0} =0$. Then $\sM_\tnil(\L,h_1(\tilde C)^{\otimes 3})=0$ by \cite[Prop. 6.1]{kim} and $\sM_\tnil(\L,M\otimes \L)=\sM_\tnil(\un,M)=0$ by (the proof of) Lemma \ref{l1}.
\end{proof}

\begin{cor}[theorem of the cube]\label{c3} Let $X_1,X_2,X_3\in \Sm^\proj$. Then the map
\[A_1^\tnil(X_1\times X_2\times X_3)\to \prod_{i<j} A_1^\tnil(X_i\times X_j)\]
induced by the three projections is injective.
\end{cor}

\begin{proof} We have $A_1^\sim(X)=\sM_\sim(\L,h^+(X))$ for any adequate relation $\sim$ and any $X\in \Sm^\proj$, and
\begin{multline*}
h^+(X_1\times X_2\times X_3) = h^+(X_1) \oplus h^+(X_2) \oplus h^+(X_3)\\
 \oplus h^+(X_1)\otimes h^+(X_2) \oplus h^+(X_1)\otimes h^+(X_3) \oplus h^+(X_2)\otimes h^+(X_3)\\
 \oplus  h^+(X_1)\otimes h^+(X_2) \otimes h^+(X_3);
\end{multline*}
\[ h^+(X_i\times X_j) = h^+(X_i)\oplus h^+(X_j)\oplus h^+(X_i) \otimes h^+(X_j)\]
hence the statement follows from  Theorem \ref{t1}.
\end{proof}

\begin{cor}[Ouriachi]\label{c2} Let $X_1,X_2,X_3\in \Sm^\proj$. Then Conjecture \ref{c1} holds for $1$-cycles on $X_1\times X_2\times X_3$ if and only if it holds for $1$-cycles on $X_1\times X_2$, $X_1\times X_3$ and $X_2\times X_3$. \qed
\end{cor}

\begin{cor}[Sebastian \protect{\cite[Th. 9]{seb}}] \label{C1} $V(X,1)$ holds for a product of curves $X=C_1\times\dots \times C_n$. 
\end{cor}

\begin{proof}  We argue by induction on $n$. The case $n\le 2$ follows from Lemma \ref{l3}. Suppose $n\ge 3$. We apply Corollary \ref{c2} to $X_1 = C_1\times \dots \times C_{n-2}$, $X_2 = C_{n-1}$ and $X_3 = C_n$.
%
%
\end{proof}

\begin{rques} 1) As Sebastian's, Ouriachi's proof relied on the smash-nilpoten\-ce of the Gross-Schoen modified diagonal $\Delta^{(3)}$ for the cube of a curve $C$\footnote{Sebastian only proved this for the modified diagonal to an $m$-th power of $C$ for $m\gg 0$, see \cite[end of introduction]{seb}; I thank the referee for pointing this out.}, which follows from the odd finite dimensionality of the tensor cube of its $h_1$. (Indeed, $\Delta^{(3)}$ viewed as a morphism from $\L$ to $h(C)^{\otimes 3}$ factors through $h^+(C)^{\otimes 3}$.)  This modified diagonal is implicit in the proof of Theorem \ref{t1}.\\
2) Unfortunately, I am not able to prove a ``theorem of the square'', i.e. $V(1,M_1\otimes M_2)$ for two reduced motives $M_1, M_2$.
\end{rques}

\subsection{Clarifying the notion of ``abelian type''}\label{s5.1} By a \emph{motivic category} we mean an additive $\otimes$-category $\sC$ provided with a $\otimes$-functor $T:\sM\to \sC$. (So $\sC$ may be $\sM_\sim$ for some other adequate equivalence $\sim$, $\sM_\sim^\o$ (see \S \ref{s5}), André's category $\sM^A$ of motivated motives \cite{A}, Deligne's category of motives for absolute Hodge cycles \cite{delmil}, or even the category of pure polarisable Hodge structures). For $X\in \Sm^\proj$, we write $h_\sC(X)$ for $T(h(X))$. We say that an object $M\in \sC$ is \emph{of abelian type} if $M$ is isomorphic to a direct summand of $n h_\sC(A)$ for some abelian variety $A$ and some integer $n>0$, where we write $nM$ for an $n$-fold direct sum $M\oplus M\oplus\dots \oplus M$.

\begin{lemma} a) The class of objects of abelian type is closed under tensor products, direct sums and direct summands; it contains the unit object and (the image of) the Lefschetz motive.\\
b) An object is of abelian type if and only if it is isomorphic to a direct summand of a motive of the form $nh_\sC(C^m)$, where $C$ is a curve.
\end{lemma}

\begin{proof} a) The claim is obvious for tensor products and direct summands; for direct sums, note that $h_\sC(A)\oplus h_\sC(B)$ is a direct summand of $2 h_\sC(A\times B)$ for two abelian varieties $A,B$. Finally, $\L$ is a direct summand of $h(E)$ for any elliptic curve $E$.

b) It suffices to show that, for any abelian variety $A$, $h_\sC(A)$ is a direct summand of $h_\sC(C^m)$ for suitable $C$ and $m$. This is classical: we choose for $C$ an ample curve on $A$ and $m=2g$, where $g=\dim A$. Then $h^1_\sC(A)$ is a direct summand of $h^1_\sC(C)$, hence $h^i_\sC(A)=S^i(h^1_\sC(A))$ is a direct summand of $h^1_\sC(C)^{\otimes i}$, which is itself a direct summand of $h_\sC(C^i)$ hence also of $h_\sC(C^{2g})$.
\end{proof}

Let $U:\sC\to \sD$ be a $\otimes$-functor. If $M\in \sC$ is of abelian type, so is $U(M)$, but the converse is not necessarily clear and may be false; for example, motives of K3 surfaces are of abelian type in $\sM^A$ \cite[Th. 7.1]{A}, but this is an open question in $\sM_\hom$. It is obviously false if $\sC=\sM$, $\sD$ is the category of $\Z$-graded $K$-vector spaces and $U$ is given by a Weil cohomology with coefficients $K$.  

If $U$ is full with locally nilpotent kernel, this converse is true by lifting idempotents; for example, a motive is of abelian type in $\sM$ if and only if it is so in $\sM_\alg$ or $\sM_\tnil$, but the same question is open for $\sM_\num$ (possibility of ``phantom motives''). A basic example is Bloch's conjecture for surfaces.

\subsection{The abelian and representable parts of a motive} Let $M\in \sM$. By Jannsen's semi-simplicity theorem, write $M_\num$ as a direct sum of simple motives: $M_\num=\bigoplus_\alpha S_\alpha$. Collect all those $S_\alpha$ which are (numerically) of abelian type (resp. of the form $h_1(A_i)(n_i)$, where $A_i$ is a simple abelian variety and $n_i\ge 0$). Call this submotive $M_\num^\ab$ (resp. $M_\num^\rep$): this is the \emph{abelian (resp. representable) part} of $M_\num$.

\begin{lemma} There exists a direct summand $M^\ab$ (resp. $M^\rep$) of $M$ lifting $M^\ab_\num$ (resp. $M^\rep_\num$); it is unique up to isomorphism.
\end{lemma}

\begin{proof} Let $S\in \sM_\num$ be a simple motive of abelian type. Write $S=(h_\num(A),p)$, where $A$ is an abelian variety and $p=p^2$ is an idempotent correspondence. By \cite[]{kim}, lift $p$ to an idempotent $\tilde p\in \End_\sM(h(A))$; this lift is unique up to conjugation. The corresponding direct summand $\tilde S=(h(A),\tilde p)$ is a lift of $S$ to $\sM$. If $S=(h_\num(B),q)$ for another abelian variety $B$, then $S$ is also a direct summand of $h(A\times B)$, which shows that $\tilde S$ does not depend on the choice of $A$, up to isomorphism.

Lift $M_\num^\ab$ to some $M^\ab\in \sM$, simple summand by simple summand. It remains to show that $M^\ab$ is isomorphic to a direct summand of $M$. Let $i:M_\num^\ab\inj M_\num$ and $\pi:M_\num\surj M_\num^\ab$ be the injection and the projection, so that $\pi \circ i= 1_{M_\num^\ab}$. Lifting $i$ to $\tilde\imath$ and $\pi$ to $\tilde\pi$, the endomorphism $\tilde\pi\circ \tilde\imath- 1_M$ is numerically equivalent to $0$, hence nilpotent, thus the claim. Same reasoning for $M^\rep$.
\end{proof}

\begin{ex} Let $X$ be a smooth complete intersection of dimension $n$ and multidegree $(a_1,\dots, a_d)$. Write $h(X)=\bigoplus_{i=0}^{2n} h_i(X)$ and $h_n(X)=\L^{n/2}\oplus p(X)$ if $n$ is even; if $n$ is odd, put $p(X):=h_n(X)$. If $\car k=0$, then $p(X)_\num$ is representable if and only if we are in the situation  of \cite[2.9]{del}. In this case, $h(X)=h(X)^\rep\oplus M$, where $M$ is a phantom motive; we cannot exclude \emph{a priori} that $M\ne 0$.
\end{ex}

\subsection{Birational motives}\label{s5}  Recall from \cite{ks} the category of pure birational motives $\sM_\sim^\o$. The following is elementary but powerful:

\begin{thm}\label{t2} Let $N_1,N_2\in \sM$. Suppose that $N_1$ becomes isomorphic to a direct summand of $N_2$ in $\sM^\o$. 
Let $n\ge 0$. If $V(M,n)$ (resp. $V^*(M,n)$) is true for any $M$,  then $V(N_2,n+1) \Rightarrow V(N_1,n+1)$ (resp. $V^*(N_2,n+1) \Rightarrow V^*(N_1,n+1)$). In particular, for $(\sim,\sim')=(\tnil, \num)$:
\begin{itemize}
\item Conjecture \ref{c1} is true for $N_1$ if $N_2$ is the motive of a curve or a surface;
\item $V(N_2,1) \Rightarrow V(N_1,1)$ and $V^*(N_2,2) \Rightarrow V^*(N_1,2)$.
\end{itemize}
\end{thm}

\begin{proof} Let $\alpha^\o:N_1^\o\iso N_2^\o$, $\beta^\o:N_2^\o\iso N_1^\o$ be such that $\beta^\o\alpha^\o=1_{N_1^\o}$. Lift them to $\alpha:N_1\to N_2$ and $\beta:N_2\to N_1$. Then $1_{N_1}-\beta\alpha$ factors through $P(1)$ for some $P\in \sM$ \cite[Cor. 2.4.3]{ks}, i.e. $1_{N_1}-\beta\alpha  = \delta\gamma$ with $\gamma:N_1\to P(1)$, $\delta:P(1)\to N_1$. 

Let $x\in \sM(\L^{n+1},N_1)$ be such that $x\sim' 0$, and suppose that $\alpha_*x\sim 0$. Then $\beta_*\alpha_* x\sim 0$, hence
\[x \sim(1_{N_1}-\beta\alpha)_*x=\delta_* \gamma_* x\]
where $\gamma_* x\in \sM(\L^n, P)$. But $\gamma_* x\sim' 0$, so $\gamma_* x\sim  0$ and finally $x\sim 0$. Dual reasoning for $x\in \sM(N_1,\L^{n+1})$.
\end{proof}

\subsection{Examples}\label{s8} Let us say that $N_1\in \sM$ is \emph{birationally of dimension $\le d$} (resp.  \emph{birationally of abelian type}) if one can choose $N_2=h(Y)$ in Theorem \ref{t2} with $\dim Y\le d$ (resp. $Y$ an abelian variety). Then

\begin{cor}\label{c4} a) If $M$ is birationally of dimension $\le 2$, then $V(M,1)$ and $V^*(M,2)$ hold (for $(\sim,\sim')=(\tnil,\num)$. In particular
\begin{description}
\item[\protect{\cite[Th. 1]{seb2}}] Conjecture \ref{c1} holds for uniruled $3$-folds.
\item[\protect{\cite[Th. 2 (A)]{seb2}}] if the MRCC-quotient of $X\in \Sm^\proj$ has dimension $\le 2$, then $V(X,1)$ and $V^*(X,2)$ hold.
\item[\protect{\cite[Prop. 3.5.3]{ks}}] If $X$ is a smooth complete intersection in $\P^r$ of multidegree $(a_1,\dots, a_d)$ with $\sum a_i \le r$, then $V(X,1)$ holds. (This covers cubic fourfolds as a very special case, \cite[Th. A]{gumu}.)
\item[{\cite[Ex. 7.3]{Kzeta}}] More generally, if $X$ is rationally chain-connected, then $V(X,1)$ holds. (This covers Fano varieties, and in particular the other examples of \cite{gumu}: I am grateful to the referee for pointing this out.)
\end{description}
b) $V(X,1)$ holds if $X$ is birationally of abelian type. In particular \cite[Th. 2 (B)]{seb2}, it holds if its MRCC quotient is of abelian type.
\end{cor}

\begin{proof} a) follows from Theorem \ref{t2} and Lemma \ref{l3}. The first consequence is obvious, and the second follows from the isomorphism of birational motives $h^\o(X)\iso h^\o(Q)$ where $Q$ is its MRCC quotient \cite[Cor. 6.8 b)]{adjoints}. Similarly, b) follows from Theorem \ref{t2} and Corollary \ref{C1}.
\end{proof}

\begin{rque} Unless mistaken, Laterveer's varieties $X$ in \cite[Th. 3.1]{laterveer} are of abelian type and such that, in a Chow-Künneth decomposition
\[h(X)\simeq \bigoplus_{i=0}^{2\dim X} h_i(X)\]
the even components $h_{2j}(X)$ are direct summands of motives $h(S_j)(j-1)$ for suitable surfaces $S_j$ (a Lefschetz twist seems to be missing in the proof of Lemma 2.15). Hence Conjecture \ref{c1} for $X$ follows from the odd finite dimensionality of $h_i(X)$ for $i$ odd, as in \cite{k-s}, and from Lemma \ref{l3} for $i$ even.
\end{rque}

In \cite[Th. 4]{seb2}, Sebastian proves that if $X$ is a smooth projective complex variety such that $\dim_\Q CH_i(X)_\Q<\infty$  for $0 \le i \le l$ for some $l$, then numerical and smash-nilpotence equivalence coincide for cycles of dimension $\le l + 1$. To recover it in the same spirit as above, we use a lemma. For 
$M\in \sM$ and $i\ge 0$, say as in \cite[Def. 2.1]{vial} 
that $CH_i(M)_\alg$ is \emph{representable} if there exists 
a smooth projective curve $C$ over $\Omega$ (not necessarily connected) and a morphism $\Gamma: h_1(C)(i)\to M_\Omega$ such that the induced morphism $\Gamma_* : CH_0(C)_\alg \to CH_i(M_\Omega)_\alg$ is surjective, where $\Omega$ is a universal domain over $k$. By loc. cit., Th. 3.4, $CH_*(M)_\alg$ is representable if and only if $M$ is isomorphic to a direct sum of Lefschetz motives and twisted $h_1$’s of abelian varieties. Vial's inductive proof actually gives the following refinement:

\begin{lemma} \label{l7.1} Let 
$M\in \sM$ and $n\ge 0$. Then the following are equivalent:
\begin{thlist}
\item $CH_i(M)_\alg$ is representable for $i=0,\dots, n-1$;
\item $M$ is of the form $\bigoplus_{i=0}^{n-1} r_i \L^i\oplus  h_1(A_i)(i)\oplus M_n(n)$ (with $M_n$ effective).
\end{thlist}
Moreover, $CH_i(M_\Omega)$ is finite dimensional for $i\le n$ if and only (ii) holds with all $A_i$'s trivial. (In particular, (i) also holds.)\qed
\end{lemma}

\begin{prop} For $M$ as in Lemma \ref{l7.1}, $V(M,j)$ holds for  $j=0,\dots, n$ for $(\sim,\sim')=(\alg,\num)$ (hence also for $(\sim,\sim')=(\tnil,\num)$).
\end{prop}

\begin{proof} Indeed, the statement is true for all summands $\L^i$ and $h_1(A_i)(i)$, and also for $M_n(n)$ (see Lemma \ref{l1} and its proof).
\end{proof}

This gives back \cite[Th. 4]{seb2} as a very special case. 

\subsection{A challenge}\label{s9} Let $\sM(V,1)$ be the full subcategory of $\sM$ determined by those $M$ such that $V(M,1)$ holds: it is thick (closed under direct sums and direct summands) by Lemma \ref{l3}. Let $\sM'$ be the smallest thick subcategory of $\sM$ containing the $M(X)$ for $X$ birationally of dimension $\le 2$ or birationally of abelian type and, for good measure, those of the form $N\otimes M_1^+\otimes M_2^+\otimes M_3^+$ as in Theorem \ref{t1}. By this theorem and Corollary \ref{c4}, we have $\sM'\subseteq \sM(V,1)$.

\begin{qns}\label{q1}\ 

1) Can one find an object of $\sM(V,1)\setminus \sM'$ in the existing literature?

2) Can one prove that $M(X)\in \sM(V,1)$ (hence Voevodsky's conjecture) for the Dwork projective hypersurface $X$ with equation $x_0^5+\dots +x_4^5=5x_0\dots x_4$? Cf. \cite[Ex. 5.13]{schoen}.
\end{qns}

\subsection{A very weak result} We finish with this proposition, which hopefully could be further upgraded.

\begin{prop} Let $X,Y\in \Sm^\proj$ be connected, and let $\alpha\in \Corr(X,Y)$ be a correspondence. Assume that $p_*\alpha=0$, where $p:X\times Y\to X$ is the first projection (e.g. that $\alpha\sim_\num 0$). Then there exists an integer $N>0$ such that the composition
\[X\by{\delta_X^N} X^N\by{\alpha^{\otimes N}} Y^N\]
is $0$, where $\delta_X^N$ is the $N$-fold diagonal.
\end{prop}

\begin{proof} For notational simplicity, we drop the coefficients $\Q$. Let $d=\dim X$, $d'=\dim Y$ so that $\alpha\in CH_d(X\times Y)=CH^{d'}(X\times Y)$. If $j:\eta\inj X$ is the inclusion of the generic point, we have $\deg j^*\alpha=0$ because $\deg:CH^{d'}(Y_\eta)=CH_0(Y_\eta)\to \Z$ is also given by $p_*$. Hence $j^*\alpha\sim_\alg 0$ and, by Theorem \ref{t0}, there exists $N_0>0$ such that $(j^*\alpha)^{\otimes_\eta N_0}=0$. In other words, $\alpha_1=\alpha^{\otimes N_0}\circ \delta_X^{N_0}$ has support in $Z\times Y^{N_0}$, where $Z\subset X$ is a proper closed subset. By \cite[Cor. 2.4.3]{ks}, this means that   $\alpha_1$ factors as
\[h(X)\by{\beta} M(1)\by{\gamma} h(Y^{N_0})\]
for some effective motive $M$. Then, for any $n>0$, $\alpha_1^{\otimes n}$ factors as
\[h(X^{n})\by{\beta^{\otimes n}} M^{\otimes n}(n)\by{\gamma^{\otimes n}} h(Y^{nN_0}).\]

If $M^{\otimes n}$ is a direct summand of $h(Z)$ for some $Z\in \Sm^\proj$, then the group $\sM(h(X),M^{\otimes n}(n))$ is a direct summand of
\[\sM(h(X),h(Z)(n))=CH_{d-n}(X\times Z)\]
which is $0$ for $n>d$. Hence the proposition.
\end{proof}


\begin{thebibliography}{23}
\bibitem{A} Y. Andr\'e {\it Pour une th\'eorie inconditionnelle des motifs}, Publ. Math. IH\'ES {\bf 83} (1996) 5--49.
\bibitem{andre} Y. Andr\'e Une introduction aux motifs, Panoramas et  synth\`eses, SMF, 2004.
\bibitem{del} P. Deligne {\it Cohomologie des intersections complètes}, Exp. XI of SGA 7, Lect. Notes in Math. {\bf 340}, Springer, 1973.
\bibitem{SGA2} A. Grothendieck {\it Application aux schémas algébriques projectifs}, Exp. XII de SGA 2, nouvelle édition, Doc. mathématiques {\bf 4}, SMF, 2005, 109--134.
\bibitem{AK} P. L. del Angel, S.-i. Kimura {\it Finite dimensional morphisms in a tensor category}, J. reine angew. Math. {\bf 656} (2011), 213--222.
\bibitem{dtens} P. Deligne {\it Cat\'egories tensorielles}, Moscow Math. J. {\bf 2} (2002), 227--248.
\bibitem{delmil} P. Deligne, J.S. Milne {\it Tannakian Categories}, in \emph{Hodge Cycles, Motives, and Shimura Varieties} LNM {\bf 900}, 1982, pp. 101-228.
\bibitem{hart} R. Hartshorne Algebraic geometry, Springer, 1977.
\bibitem{jann3} U. Jannsen {\it Motivic sheaves and fiiltrations on Chow groups}, {\it in} Motives, Proc. Symp. pure Math. {\bf 55} (I), 245--302.
\bibitem{jannsen} U. Jannsen {\it Equivalence relations on algebraic cycles}, {\it in} The arithmetic and Geometry of
algebraic cycles, proc. Nato conference, Banff {\bf 98}, Nato series {\bf 548} (2000), 225--260.
\bibitem{Kzeta} B. Kahn {\it Zeta functions and motives}, Pure appl. Math. Quarterly {\bf 5} (2009),  507--570.
\bibitem{adjoints} B. Kahn {\it Motifs et adjoints}, Rend. Sem. mat. Univ. Padova {\bf 139} (2018), 77--128.
\bibitem{tunis} B. Kahn {\it Albanese kernels and Griffiths groups}, Tunis J. Math. {\bf 3} (2021), 589--656.
\bibitem{essential} B. Kahn {\it Sur la conjecture de Tate pour les diviseurs},  Ess. Numb. Theory {\bf 2} (2023), 83--92.
\bibitem{kmp} B. Kahn, J. Murre, C. Pedrini {\it On the transcendental part of the motive of a surface}, {\it in} Algebraic cycles and motives, London Math. Soc. Lecture Note Ser., {\bf 344} (2), Cambridge Univ. Press, 2007, 143--202. 
\bibitem{k-s} B. Kahn, R. Sebastian {\it Smash-nilpotent cycles on abelian 3-folds}, Math. Res. Lett. {\bf 16} (2009), 1007--1010.
\bibitem{ks} B. Kahn, R. Sujatha {\it Birational motives, I: pure birational motives}, Ann. $K$-theory {\bf 1} (2016), 379--440.
\bibitem{kim} S.-I. Kimura {\it Chow groups are finite dimensional, in some sense}, Math. Ann. {\bf 331} (2005), 173--201.
\bibitem{kunn} K. Künnemann {\it On the Chow motive of an abelian scheme}, {\it in} Motives (Seattle, 1991), Proc. Symp. pure Math. {\bf 55} (1), AMS, 1994, 189--205.
\bibitem{laterveer} R. Laterveer {\it Some new examples of smash-nilpotent algebraic cycles}, 
Glasg. Math. J. {\bf 59} (2017), 623--634. 
\bibitem{lieb} D. Lieberman {\it Numerical and homological equivalence of algebraic cycles on Hodge manifolds}, Amer. J. Math {\bf 90} (1968), 366--374.
\bibitem{matsusaka} T. Matsusaka {\it The criteria for algebraic equivalence and the torsion group}, Amer. J. Math. {\bf 79} (1957), 53--66.
\bibitem{mumford} D. Mumford Abelian varieties, Tata Inst. Fund. Research Studies in Mathematics {\bf  5}, Bombay, 1970.
\bibitem{nori} M. Nori {\it Algebraic cycles and Hodge-theoretic connectivity}, Invent. Math. {\bf 111} (1993), 349--373.
\bibitem{gumu} M. Ornaghi, L. Pertusi {\it Voevodsky’s conjecture for cubic fourfolds and Gushel–Mukai fourfolds via noncommutative K3 surfaces}, J. Noncommut. Geom. {\bf 13} (2019), 499--515.
\bibitem{os} P. O'Sullivan, letter to Y. André and B. Kahn, 12 May 2002.
\bibitem{schoen} C. Schoen {\it Varieties dominated by product varieties}, Int. J. Math. {\bf 7} (1996), 541--571.
\bibitem{seb} R. Sebastian {\it Smash nilpotent cycles on varieties dominated by products of curves}, Compos. Math. {\bf 149} (2013), 1511--1518.
\bibitem{seb2} R. Sebastian {\it Examples of smash nilpotent cycles on rationally connected varieties},  J. Algebra {\bf 438} (2015), 119--129.
\bibitem{vial} C. Vial {\it Pure motives with representable Chow groups}, C. R. Acad. Sci. Paris {\bf 348} (2010), no. 21-22, 1191--1195.
\bibitem{voi} C. Voisin {\it Remarks on zero-cycles on self-products of varieties}, {\it in} Moduli of vector bundles (Sanda, 1994; Kyoto, 1994), Lect. Notes in Pure and Appl. Math. {\bf 179}, Dekker, 1996, 265--285.
\bibitem{voe} V. Voevodsky {\it A nilpotence theorem for cycles algebraically equivalent to zero}, 
IMRN {\bf 1995}, 187--198. 
\end{thebibliography}
\end{document}